\documentclass[12pt]{amsart}
\usepackage{amsfonts}
\usepackage{amssymb}
\usepackage[letterpaper, left=2.5cm, right=2.5cm, top=2.5cm,
bottom=2.5cm,dvips]{geometry}
\usepackage{verbatim}
\usepackage{graphicx}
\usepackage{amsmath}

\newtheorem{theorem}{Theorem}
\theoremstyle{plain}

\newtheorem{conjecture}{Conjecture}

\newtheorem{proposition}{Proposition}
\newtheorem{question}{Question}

\numberwithin{equation}{section}

\usepackage{array,multirow} 

\usepackage{xcolor}

\usepackage[backref]{hyperref}
\hypersetup{
	colorlinks,
	linkcolor={blue!60!black},
	citecolor={red!60!black},
	urlcolor={red!60!black}
}

\begin{document}
	\title[Square-free Extensions of Words]{Square-free Extensions of Words}

	\author{Jaros\l aw Grytczuk}
	\address{Faculty of Mathematics and Information Science, Warsaw University
		of Technology, 00-662 Warsaw, Poland}
	\email{j.grytczuk@mini.pw.edu.pl}
	
	\author{Hubert Kordulewski}
	\address{Faculty of Mathematics and Information Science, Warsaw University
		of Technology, 00-662 Warsaw, Poland}
	\email{apostatajulian@gmail.com}
	
	\author{Bart\l omiej Pawlik}
	\address{Institute of Mathematics, Silesian University of Technology, 44-100 Gliwice, Poland}
	\email{bpawlik@polsl.pl}
	
	\thanks{Research supported by the National Science Center of Poland, grant 2015/17/B/ST1/02660.}

	\begin{abstract}
A word is \emph{square-free} if it does not contain nonempty factors of the form $XX$. In 1906 Thue proved that there exist arbitrarily long square-free words over a $3$-letter alphabet. It was proved recently \cite{GrytczukKN} that among these words there are infinitely many \emph{extremal} ones, that is, having a square in every single-letter extension.

We study diverse problems concerning extensions of words preserving the property of avoiding squares. Our main motivation is the conjecture stating that there are no extremal words over a $4$-letter alphabet. We also investigate a natural recursive procedure of generating square-free words by a single-letter right-most extension. We present the results of computer experiments supporting a supposition that this procedure gives an infinite square-free word over any alphabet of size at least three.

	\end{abstract}
	
	\maketitle
	
	\section{Introduction}
	
	A \emph{square} is a finite nonempty word of the form $XX$. For instance, the word $\mathtt{hotshots}$ is a~square with $X=\mathtt{hots}$. A word $W$ \emph{contains} a square if it can be written as $W=UXXV$ for some words $U,V$, and a nonempty word $X$. A word is \emph{square-free} if it does not contain any squares. For instance, the word $\mathtt{repetition}$ contains the square $\mathtt{titi}$, while $\mathtt{recreation}$ is square-free.
	
	 It is easy to check that there are no binary square-free words of length greater than $3$. However, there exist ternary square-free words of any length, as proved by Thue in \cite{Thue} (see~\cite{BerstelThue}). This result is the starting point of Combinatorics on Words, a wide discipline with many exciting problems, deep results, and important applications (see \cite{AlloucheShallit,BeanEM,BerstelPerrin,Currie TCS,GrytczukDM,Lothaire,Lothaire Algebraic}).
	
	In this paper we study problems concerning \emph{square-free extensions} of words, a new concept introduced recently in \cite{GrytczukKN}. Let $\mathbb A$ be a fixed alphabet and let $W$ be a finite word over $\mathbb A$. The~set of all finite words over $\mathbb{A}$ is denoted by~$\mathbb{A}^\star$. An \emph{extension} of $W$ over $\mathbb{A}$ is any word of the form $W'xW''$, where $x\in \mathbb A$ and $W=W'W''$. For instance, the word $\mathtt{bear}$ is an extension of the word $\mathtt{bar}$ by inserting the~letter $\mathtt{e}$ between letters $\mathtt{b}$ and $\mathtt{a}$. A square-free word $W$ is called \emph{extremal} over $\mathbb A$ if there is no square-free extension of $W$. For instance, the word $$H=\mathtt {1231213231232123121323123}$$is the shortest extremal word over alphabet $\{\mathtt{1,2,3}\}$. This means that inserting any letter from the alphabet $\{\mathtt{1,2,3}\}$ at any position in the word $H$, including the beginning as well as the end of $H$, results with a square.
	
	A natural intuition is that extremal words should be rare or even should have bounded length. However, in the case of a $3$-letter alphabet, the said intuition turned out to be untrue.
	
	\begin{theorem}[Grytczuk, Kordulewski, Niewiadomski \cite{GrytczukKN}]\label{Theorem Extremal 3 Words}
	There exist infinitely many extremal words over a $3$-letter alphabet.
	\end{theorem}

The proof of this theorem is by a recursive construction whose validity is partially based on computer verifications. In \cite{MolRampersad} Mol and Rampersad determined all positive integers $k$ for which there exist extremal ternary words of length exactly $k$. In particular, they proved that for every $k\geqslant87$ there exists an extremal ternary word of length $k$.

One may naturally wonder what the case is for larger alphabets. Actually, we do not know if there are any extremal words over a $4$-letter alphabet. The following conjecture was stated in \cite{GrytczukKN}.

\begin{conjecture}\label{Conjecture 4 Extremal}
Every square-free word over a $4$-letter alphabet can be extended to a square-free word.
\end{conjecture}

Actually, we do not even know if there exist any constants $n$ and $N$ such that every square-free word over a $n$-letter alphabet of length at least $N$ can be extended.

In the forthcoming sections we shall present some observations and results of computer experiments inspired by the conjecture above.

\section{Nonchalant words}
\subsection{The main conjectures}
The problem of extremal square-free words is connected to the following recursive construction.

Given a fixed ordered alphabet $\mathbb A$, we start with the first letter from $\mathbb A$ and continue by inserting the earliest possible letter at the rightmost position of the actual word so that the~new word is square-free. For instance, for the alphabet $\{\mathtt{1,2,3}\}$ this greedy procedure starts with the following sequence of square-free words: $$\mathtt {1, 12, 121, 1213, 12131, 121312, 1213121, 12131231}.$$ The last word was obtained by inserting $\mathtt 3$ at the penultimate position of the previous word.

We conjecture that the aforementioned procedure never stops. To state it formally, let us define recursively a sequence of \emph{nonchalant words} $N_i$ over the alphabet $\mathbb A_n=\{\mathtt{1,2,\dots,n}\}$ by putting $N_1=\mathtt 1$, and letting $N_{i+1}=N_i'xN_i''$ to be a square-free extension of $N_i$ such that $N_i''$ is the shortest possible suffix of $N_i$ and $x\in \mathbb A_n$ is the earliest possible letter.

\begin{conjecture}\label{Conjecture Nonchalant}
	The sequence of nonchalant words over $\mathbb A_n$ is infinite for every $n\geqslant 3$.
\end{conjecture}

In other words, we believe that the nonchalant algorithm never produces an extremal word. The results of computer experiments supports this conjecture. For instance, for $n=3$ a nonchalant word of length $10000$ was obtained. Moreover, the algorithm never moved back by more that $20$ positions (see Appendix \ref{apA} for details). Therefore the following conjecture seems also plausible.

\begin{conjecture}\label{Conjecture Nonchalant Infinite}
	The sequence of nonchalant words over $\mathbb A_n$ converges to an infinite word $\mathcal{N}_n$ for every $n\geqslant 3$.
\end{conjecture}

Here are the first $70$ terms of the presumably infinite limit word for $n=3$: $$\mathcal{N}_3=\mathtt{1213123132123121312313231213123212312131231321231213123212312132123132...}$$

\subsection{Playing with initial words}
The above version of nonchalant algorithm with the two corresponding conjectures were stated in~\cite{GrytczukKN}. Our numerical experiments led us to introduce more general approach.

Firstly, let us consider the nonchalant algoritm which not necessarily starts with the~letter~$\mathtt{1}$. Namely, let the~nonchalant word $N_1$ be some square-free word over considered alphabet. From now on $N_1$ will be called the {\it initial word} of the nonchalant algorithm. The results of testing 10 000 iterations of the nonchalant algorithm for various initial words are prezented in Table \ref{nai}. First column contains the initial words, while the other columns shows how many times the nonchalant algorithm moved back by the given number of positions (the column initialized by 0 shows how many times algoritm puts a letter at the rightmost position, by 1 - at the penultimate position, etc.). What is worth noticing, experimental results suggestes, that the outcomes of the algorithm contains a lot of similarities. For example, for each considered initial word, the nonchalant algorithm inserts a new letter 33 to 36 times right before the suffix of length 4 (more results are presented in Appendix \ref{apA}).

\begin{table}\centering
	\begin{tabular}{|c||ccccccccccc|}\hline
	    $N_1$&0&1&2&3&4&7&9&12&14&15&20\\ \hline
	    $\mathtt{1}$&9457&310&184&1&33&11&1&0&0&1&2\\
	    $\mathtt{2}$&9457&309&186&1&33&11&0&0&1&0&2\\
	    $\mathtt{3}$&9457&307&185&0&34&13&1&0&1&0&2\\
	    $\mathtt{13}$&9454&310&185&0&34&13&1&0&1&0&2\\
	    $\mathtt{23}$&9458&307&185&1&34&11&0&1&1&0&2\\
	    $\mathtt{32}$&9458&309&185&1&33&11&0&0&1&0&2\\
	    $\mathtt{3213}$&9455&309&185&0&34&13&1&0&1&0&2\\
	    $\mathtt{2313}$&9457&310&185&0&34&11&1&0&1&0&1\\
	    $\mathtt{32132}$&9460&307&182&1&33&14&0&0&1&0&2\\
	    $\mathtt{2313213}$&9455&309&185&0&34&13&1&0&1&0&2\\
	    $\mathtt{231323}$&9457&308&183&0&35&13&1&0&1&0&2\\
	    $\mathtt{32132313}$&9461&307&182&1&36&11&1&0&1&0&0\\
	    $\mathtt{321323132}$&9461&307&182&1&36&11&1&0&1&0&0\\
	    $\mathtt{32132313213}$&9455&309&185&0&34&13&1&0&1&0&2\\
	    \hline
	\end{tabular}
\vskip0.5cm
		\caption{The number of letters omitted through the single iteration of the~nonchalant algorithm for various initial words (10 000 iterations).}\label{nai}
\end{table}

\subsection{Nonchalant words over four letters}
In case of a $4$-letter alphabet the situation looks even more exciting. In our experiments for the initial word $\mathtt{1}$ the~nonchalant algorithm never moved back by more than one position. To be precise, through 50 000 iterations algorithm extended the word on the penultimate position only 33 times (in other cases, algorithm extended the word at the last posistion).

The square-free word $W$ over given alphabet $\mathbb{A}$ is \emph{almost extremal} if for every nonempty words $W'$ and $W''$ such that $W=W'W''$, the word $W'xW''$ contains a square for any $x\in\mathbb{A}$.  Given that, let us consider another variant of the nonchalant algorithm. Namely, let $\mathtt{12}$ be the initial word of the nonchalant algorithm over $\mathbb{A}_4$ and let us allow to extend the word only at internal positions (in this variant of nonchalant algorithm extending the word on the rightmost position is forbidden). Such procedure starts with the following sequence of square-free words: $$\mathtt {12, 132, 1312, 13142, 131412, 1314132, 13141312, 131413212}.$$
The last word was obtained by inserting $\mathtt{2}$ right before the suffix $\mathtt{12}$. Through 50~000 iterations, the~algorithm never moved back by more than two positions (in this case the number of iterations in which algorithm moved back by two positions is approximately equal to 10\% of all iterations).

\subsection{Extensions close to the ends}

These experiments led us to the following two suppositions that perhaps: (1) every quaternary square-free word can be extended at the end or at the penultimate position, and (2) every quaternary square-free word (of lenght at least 3) can be extended at one of the two rightmost internal positions. However, both suppositions turned out not to be true.

\begin{proposition}\label{Proposition 4 Non-entendable} There exists a quaternary square-free word $S$ which cannot be extended, neither at the last, nor at the penultimate position. 
\end{proposition}
\begin{proof} 
	Let $A = \mathtt{1213121}$ and $B = \mathtt{121312}$. Next, let $Y = \mathtt{3}B\mathtt{4}$ and $Z = \mathtt{41}YA\mathtt{4}YB\mathtt{341}$. Finally, put $S=ZYA\mathtt{4}YA$, which gives the word: $$S=\mathtt{4231213124121312143121312412131231423121312412131214312131241213121}.$$
	It can be verified (by a computer) that $S$ is indeed square-free. Now, $A\mathtt{x}$ contains a square for every $\mathtt{x}\in \{\mathtt{1,2,3}\}$. Also $S\mathtt{4}=Z(YA\mathtt{4})(YA\mathtt{4})$ is not square-free. For the penultimate position it suffices to check only letters $\mathtt{3}$ and $\mathtt{4}$. So, the suffix $A$ in $S$ will turn to one of the~forms, $B\mathtt{31}$ or $B\mathtt{41}$, respectively. In the latter case we get the word
	\begin{equation}
	ZYA\mathtt{4}YB\mathtt{41}=ZYA\mathtt{4}\mathtt{3}(B\mathtt{4})(B\mathtt{4})\mathtt{1}.
	\end{equation}
	 In the former case we obtain
	 \begin{equation}
	 ZYA\mathtt{4}YB\mathtt{31}=(\mathtt{41}YA\mathtt{4}YB\mathtt{3})(\mathtt{41}YA\mathtt{4}YB\mathtt{3})\mathtt{1}.
	 \end{equation}
	 The assertion is proved.
	\end{proof}

In the case of the second supposition, we will present more general result. We use the well known \emph{Zimin words} $Z_n$, defined recursively over an infinite alphabet of variables $\{x_1,x_2,x_3,\dots\}$ by $Z_1=x_1$ and $Z_{n}=Z_{n-1}x_nZ_{n-1}$ for every $n\geqslant2$. In the following proposition, the construction in the proof is more clear when we analyse the \emph{leftmost} internal positions instead of the \emph{rightmost} ones. Obviously, the result holds in the latter case as well.

\begin{theorem}\label{Theorem Zimin}
	For every natural numbers $n$ and $t$, with $n\geqslant4$ and $1\leqslant t<n$, there exists a~square-free word $W$ over the alphabet $\mathbb{A}_n$ which is non-extendable at any of its $t$ inner leftmost positions. 
\end{theorem}

\begin{proof}
	Let $A=\mathtt{12\ldots n}$ be a word over alphabet $\mathbb{A}_n$. This word has exactly  $N=(n-~1)(n-2)$ distinct internal square-free extensions. Let us consider the Zimin word $Z_{N}$ over the~alphabet~$\mathbb{A}_{N}$ and the homomorphism
	$\varphi:\mathbb{A}_{N}^*\to\mathbb{A}_n^*,$
	such that the image of every letter of the~alphabet~$\mathbb{A}_{N}$ is a unique internal extension of the word $A$ assigned in a natural way as follows:
	\begin{align*}
		\varphi(\mathtt{1})=&\,\mathtt{1\mathbf{3}23\ldots n}\\
		\varphi(\mathtt{2})=&\,\mathtt{1\mathbf{4}23\ldots n}\\
		\vdots&\\
		\varphi(\mathtt{n-2})=&\,\mathtt{1\mathbf{n}23\ldots n}\\
		\varphi(\mathtt{n-1})=&\,\mathtt{12\mathbf{1}3\ldots n}\\
		\varphi(\mathtt{n})=&\,\mathtt{12\mathbf{4}3\ldots n}\\
		\vdots&\\
		\varphi(\mathtt{N})=&\,\mathtt{123\ldots (n-1)\mathbf{(n-2)}n}.\\
	\end{align*}
For every $\mathtt{i}\in \mathbb{A}_N$ the word $\varphi(\mathtt{i})$ has a unique factor $\mathtt{je(j+1)}$, where $\mathtt{e}$ is the inserted letter that extended the word $A$. Moreover, if $\mathtt{j}\neq\mathtt{1}$, then $\varphi(\mathtt{i})$ has also a unique factor $\mathtt{(j-1)je}$ and if $\mathtt{j+1}\neq\mathtt{n}$, then $\varphi(\mathtt{i})$ has a unique factor $\mathtt{e(j+1)(j+2)}$.
	
		Let us assume that the word $\varphi(Z_{N})$ contains a square $XX$. It is not hard to verify that for any $\mathtt{x,y}\in\mathbb{A}_{N}\backslash\{\mathtt{1}\}$, $\mathtt{x}\neq \mathtt{y}$, the words $\varphi(\mathtt{1x})$, $\varphi(\mathtt{x1})$, $\varphi(\mathtt{1x1})$, $\varphi(\mathtt{x1y})$, $\varphi(\mathtt{1x1y})$ and $\varphi(\mathtt{x1y1})$ are square-free. It follows that the length of $XX$ has to be greater than $3n+2$ and $XX$ has to contain a block $\varphi(\mathtt{u})$ for some $\mathtt{u}\neq\mathtt{1}$. Moreover, this block is unique in $XX$ since every factor of a Zimin word contains a unique single letter of the greatest value (in that factor). In consequence the square $XX$ contains a unique factor $\mathtt{je(j+1)}$, which must occupy the~middle of the word $XX$. This fact gives us two possible cases for the form of the~word~$X$, namely
		$$X=\mathtt{(j+1)}Y\mathtt{je}\ \mbox{ or }\ X=\mathtt{e(j+1)}Y\mathtt{j},$$
		for some nonempty word $Y$. We may also assume that $\mathtt{j}=1$ or $\mathtt{j+1}=\mathtt{n}$, since otherwise one of the parts of the square $XX$ would have to contain another unique factor, which is clearly impossible.
		
		Let us consider the case $\mathtt{j}=\mathtt{1}$. Then we have $$\varphi(\mathtt{u})=\mathtt{1e23\ldots n}.$$ To avoid a second unique factor, the word $X$ has to be of form
		$$X=\mathtt{e23\ldots n}B\mathtt{1}$$
		for some nonempty word $B$, and so
		$$XX=\mathtt{e23\ldots n}B\varphi(\mathtt{u})B\mathtt{1}.$$
		Let us notice that $XX$ is a factor of the word $$\varphi(\mathtt{u})B\varphi(\mathtt{u})B\mathtt{1},$$
		since the only word of form $\varphi(\mathtt{y})$ with the suffix $=\mathtt{e23\ldots (n-1)n}$ is $\varphi(\mathtt{u})$. Between any occurences of the same letter in Zimin word there is a letter of greater value, so the word $Z$ contains a factor $\varphi(\mathtt{z})$ for some $\mathtt{z}$ greater than $\mathtt{x}$, and this fact creates a contradiction.
		
		The reasoning in the case $\mathtt{j+1}=\mathtt{n}$ goes analogously. Thus we have proved that the word $\varphi(Z_{N})$ is square-free.
		
		 In a similar way one may prove that the word $A\varphi(Z_{N})$ is also square-free. By the structure of Zimin words, it follows that inserting a single letter on one of the internal positions of the~prefix~$A$ in the word $A\varphi(Z_{N})$ generates a square.
\end{proof}

	The above results lead naturally to the following question.
	
	\begin{question}\label{Problem t Last Positions}
	Is it true that there is some constant $t\geqslant3$ such that every quaternary square-free word can be extended at some of its rightmost $t$ positions?
	\end{question}

We can only prove that the answer is negative over $5$-letter alphabet for sufficiently large~$t$ if we omit the~very last position in the process of extension.

\begin{proposition}\label{Proposition Last t Inner Positions} For every $t\geqslant87$ there exists a square-free word over a $5$-letter alphabet which is non-extendable at any of its $t$ inner rightmost positions. 
\end{proposition}

\begin{proof} Let $A=a_1a_2\cdots a_t$ be any extremal word of length $t$ over alphabet $\{\mathtt{1,2,3}\}$. Let $S=\mathtt{4}A\mathtt{5}$. Consider all possible extensions of the word $S$ by letters $\mathtt{4}$ or $\mathtt{5}$ at all inner positions of $A$. There are exactly $2(t-1)$ such words and we may enumerate them as $A_1,A_2,\dots,A_{2t-2}$. Let us denote also $A_{2t-1}=\mathtt{4}A\mathtt{45}$. Now, consider the word $W$ obtained as the effect of a substitution $x_i=A_i$ of words $A_i$ to the corresponding letters of the Zimin word $Z_{2t-1}$. Finally, let us denote $P=WS$.
	
	We claim that the word $P$ satisfies the assertion of the proposition. Indeed, consider any extension of $P$ at any of its $t$ final inner positions. If the inserted letter is from the alphabet $\{\mathtt{1,2,3}\}$, then we get a square by the extremality of $A$. Otherwise, if the inserted letter is from $\{\mathtt{4,5}\}$, then the suffix $S$ of $P$ becomes one of the words $A_i$ and we get a square by the~structure of the Zimin word $Z_{2t-1}$.
	
	It is also not hard to demonstrate that the word $P$ is indeed square-free, by a reasoning similar to the one in the previous proof.
	\end{proof}

\section{The number of square-free extensions}

\subsection{The square-free potential}
The problem of square-free extensions leads to some naturally defined functions on words. For instance, given a square-free word $W$ over alphabet $\mathbb{A}$, let $\AE (W)$ and $\ae (W)$ denote\footnote{Such designation of the function, borrowed from the norwegian alphabet, was chosen in order to honor Axel Thue.}, respectively, the number of different square-free extensions and the number of different \emph{internal} square-free extensions of $W$. Such functions will be called a \emph{square-free potential} and an~\emph{internal square-free potential} of the word $W$. Given that, we can rephrase some definitions in the terms of square-free potentials: the square-free word $W$ is \emph{extremal} if $\AE(W)=0$, \emph{almost extremal} if $\ae(W)=0$, and \emph{maximal} if $\AE(W)=\ae(W)$.

Let us notice that for every square-free word $W$ over the alphabet $\mathbb{A}_n$, the inequality $$\AE(W)\leqslant\ae(W)+2(n-1)$$ holds.

As we already know, $\AE(W)=0$ for infinitely many square-free ternary words. But how large this function can be for words of length $n$?

Let $\mathcal{S}_n$ denote the set of all finite square-free words over the alphabet $\mathbb{A}_n$. Let $\AE_n(k)$ and $\ae_n(k)$ be the maximum values of $\AE(W)$ and $\ae(W)$ for words of length $k$ in $\mathcal{S}_n$. Clearly, $\AE_3(k)\leqslant k+3$ and $\ae_3(k)\leqslant k-1$ for all $k\geqslant 1$, by definition (every ternary square-free word can be potentially extended at every internal position by just one letter, and at the border positions by 2 distinct letters each). However, notice that the number of internal positions where such word may be extended is limited by the number of \emph{palindromes} it contains. Indeed, any palindrome inside a square-free word must contain a factor $\mathtt{xyx}$ in the middle, which simply cannot be extended at the left nor at the right end. Such palindromes occur very often in square-free ternary words, at least once in every factor of length $7$, which gives the following bound.

\begin{proposition}\label{Proposition AE_3(k)}
For every sufficiently large $k$, we have $\AE_3(k)\leqslant \frac{5}{7}k$.
\end{proposition}

However, the numerical results presented in Table \ref{AE3ae3} suggest that a better upper bound may be found. 

Let us consider a square-free word $M$ of length 35 such that
$$M=\mathtt{1\_213\_123\_132\_312\_321\_231\_213\_123\_132\_312\_321\_2},$$
where the symbol $\_$ stands for an extendable position of the word $M$. What is worth of notice, the prefixes $M_k$ of lenghts $k\geqslant7$ of the word $M$ has the maximal possible value of the function $\ae$ for given $k$. Moreover, the internal positions on which the words $M_k$ are extendable coincide with internal extendable positions of the word $M$.

\begin{table}\centering
	\begin{tabular}{|c||cccccccccccccccc|}\hline
		$k$&3&4&5&6&7&8&9&10&11&12&13&14&15&16&17&18\\
		$\ae_3(k)$&2&3&4&3&2&3&3&3&4&4&4&5&5&5&6&6\\
		$\AE_3(k)$&6&7&6&6&6&6&6&6&6&7&7&7&8&8&8&9\\\hline
		$k$&19&20&21&22&23&24&25&26&27&28&29&30&31&32&33&34\\
		$\ae_3(k)$&6&7&7&7&8&8&8&9&9&9&10&10&10&11&11&11\\
		$\AE_3(k)$&9&9&10&10&10&11&11&11&12&12&12&13&13&13&14&14\\\hline
		$k$&35&36&37&38&39&40&41&42&43&44&45&46&47&48&49&50\\
		$\ae_3(k)$&12&11&11&11&12&12&12&12&12&13&13&13&14&14&14&15\\
		$\AE_3(k)$&14&15&14&15&16&15&14&15&15&15&16&16&16&17&17&17\\\hline
	\end{tabular}
	\vskip0.5cm
	\caption{Values of the functions $\ae_3(k)$ and $\AE_3(k)$ for $k\leqslant50$.}\label{AE3ae3}
\end{table}

The construction of the nonchalant algorithm suggests that answering the question concerning the values of function $\ae$ for the nonchalant words seems plausible. The numerical results shows that for the first 1000 iterations of the nonchalant algorithm, the values of $\ae$ are slightely increasing in such manner, that when the new maximal value is obtained for the~nonchalant word $N_i$, then the value of $\ae$ for preceding words is never less than $\ae(N_i)-2$ (see Tables \ref{aen} and \ref{aemax}). 

\begin{table}\centering
	\begin{tabular}{|c||ccccccccccccccccccc|}\hline
		$i$&2&3&4&5&6&7&8&9&10&11&12&13&14&15&16&17&18&19&20\\
		$\ae(N_i)$&1&2&2&2&2&2&3&3&3&3&3&2&2&1&1&1&2&2&2\\\hline
		$i$&21&22&23&24&25&26&27&28&29&30&31&32&33&34&35&36&37&38&39\\
		$\ae(N_i)$&3&3&3&3&3&4&4&4&4&4&4&5&4&4&4&4&5&5&5\\\hline
	\end{tabular}
	\vskip0.5cm
	\caption{Values of the function $\ae$ for nonchalant words $N_i$ for $i\leqslant39$.}\label{aen}
\end{table}

\begin{table}\centering
	\begin{tabular}{|c||ccccccccccccccc|}\hline
		$i$&2&3&8&26&32&40&46&64&79&100&108&111&117&135&172\\
		$\ae(N_i)$&1&2&3&4&5&6&7&8&9&10&11&12&13&14&15\\\hline
		$i$&175&183&189&222&243&251&254&260&279&286&314&338&346&352&370\\
		$\ae(N_i)$&16&17&18&19&20&21&22&23&24&25&26&27&28&29&30\\\hline
		$i$&385&406&414&417&423&445&469&477&489&496&524&548&556&562&580\\
		$\ae(N_i)$&31&32&33&34&35&36&37&38&39&40&41&42&43&44&45\\\hline
		$i$&595&616&624&627&633&655&687&706&712&737&740&743&764&779&800\\
		$\ae(N_i)$&46&47&48&49&50&51&52&53&54&55&56&57&58&59&60\\\hline
		$i$&808&811&817&835&850&872&875&878&881&902&917&938&967&973&997\\
		$\ae(N_i)$&61&62&63&64&65&66&67&68&69&70&71&72&73&74&75\\\hline
	\end{tabular}
	\vskip0.5cm
	\caption{Indexes $i<1000$ for which the nonchalant procedure gave new maximal values of $\ae(N_i)$.}\label{aemax}
\end{table}

This fact led us to the following version of Conjecture \ref{Conjecture Nonchalant}.

\begin{conjecture}
	If $\ae(N_1)>2$ for the starting word $N_1$ of the nonchalant algorithm, then the~respective sequence of nonchalant words is infinite.
\end{conjecture}

\subsection{The square-free potential of Zimin words}

Let us consider the square-free potential of Zimin words $Z_n$. Since Zimin words are non-extendable on the external positions, we have $\AE(Z_n)=\ae(Z_n)$. It is easy to verify, that $\AE(Z_1)=\AE(Z_2)=0$ and $\AE(Z_3)=2$. Let $n\geqslant4$. From the construction $Z_{n}=Z_{n-1}\mathtt{n}Z_{n-1}$ we get that the only extentions of $Z_n$ by the letters $\mathtt{1,2,\ldots,n-1}$ are those induced by the~prefix or sufix $Z_{n-1}$ (the \emph{new} letter $\mathtt{n}$ in the center of the word $Z_n$ provides that there is no square that contains this letter). Thus
$$\AE(Z_{n})=2\cdot\AE(Z_{n-1})+t(n),$$
where $t(n)$ stands for the number of different square-free extensions of the word $Z_n$ by inserting the letter $\mathtt{n}$. Let us recall that the extention of $Z_n$ by the letter $\mathtt{n}$ would generate a~square if and only if we insert this letter right after the last appeareance of any other letter in the prefix $Z_{n-1}\mathtt{n}$ or, analogously, right before the first appeareance of any other letter in the sufix $\mathtt{n}Z_{n-1}$. Since there are $2^n-2$ internal positions in the word $Z_n$, we have $$t(n)=(2^n-2)-2(n-1).$$ This leads us to the following

\begin{proposition}
	Let $Z_n$ be a Zimin word over an alphabet $\mathbb{A}_n$. Then	$\AE(Z_1)=\AE(Z_2)=0$, $\AE(Z_3)=2$ and
	\begin{equation}
	\AE(Z_n)=2^n-2n+2\cdot\AE(Z_{n-1}),
	\end{equation}
 for $n\geqslant4$.
\end{proposition}

Before we came up with the presented formula, our computer calculations gave us the~following sequence of the square-free potentials of Zimin words:
$$0,0,2,12,46,144,402.$$
Excluding the first two 0's, there is only one sequence in the OEIS which is initialized by such integers (see \cite{OEIS}). Description of the sequence suggests, that it has nothing to do with combinatorics on words. Needless to say, the next term of the OEIS sequence is equal to 1040 and, as it turned out, the next term of our sequence is equal to 1044. Thus, for the~brief moment, the authors became victims of the well known \emph{Strong Law of Small Numbers}.

\section{Final discussion}

Let us conclude the paper with a few general remarks and open problems. First notice that one may consider extremal words and nonchalant words with respect to any \emph{avoidable pattern}. For instance, one natural generalization of squares is that of \emph{$k$-powers}, which are words of the form $XX\cdots X$ consisting of $k$ copies of any nonempty word $X$. It was already proved by Thue \cite{Thue2} (see \cite{BerstelThue}, \cite{Lothaire}) that there exist infinitely many \emph{cube-free} words over a $2$-letter alphabet. Is the sequence of cube-free nonchalant \emph{binary} words infinite? Is it true that every cube-free \emph{ternary} word is extendable?

Similar questions can be asked for \emph{overlap-free} words (avoiding factors of the form $xWxWx$, where $x$ is a single letter and $W$ is a word). By the famous result of Thue \cite{Thue2} there exist binary overlap-free words of any length. Mol, Rampersad, and Shallit \cite{MolRampersadShallit} proved recently that there are infinitely many extremal overlap-free binary words. Actually they determined precisely the possible lengths of such words.

To state our main conjectures for general patterns let us recall briefly some basic notions of pattern avoidance (see \cite{BeanEM,Lothaire Algebraic}). Let $\mathbb V$ be an alphabet of variables. A \emph{pattern} $P=p_1p_2\dots p_r$, with $p_i\in \mathbb V$, is any nonempty word over $\mathbb V$. A~word $W$ \emph{realizes} a pattern $P$ if it can be split into nonempty factors $W=W_1W_2\dots W_r$ so that $W_i=W_j$ if and only if $p_i=p_j$, for all $i,j=1,2,\dots,r$. A~word~$W$ \emph{avoids} a pattern $P$ if no factor of $W$ realizes $P$. For instance, a square-free word avoids a pattern $P=xx$. A pattern $P$ is \emph{avoidable} if there exist arbitrarily long words avoiding $P$ over some finite alphabet. The least size of such alphabet is denoted as $\mu(P)$ and called the \emph{avoidability index} of $P$. A complete characterizations of avoidable patterns was provided independently by Zimin \cite{Zimin} and Bean, Ehrenfeucht and McNulty \cite{BeanEM}.

Now, given a fixed pattern $P$, we may define \emph{extremal $P$-free} words and \emph{$P$-nonchalant words} analogously as in the case of squares. The following conjectures seem worth experimentation.

\begin{conjecture}\label{Conjecture Extremal Patterns}
	For every avoidable pattern $P$, there are no extremal $P$-free words over alphabet of size $\mu(P)+1$.
\end{conjecture}

\begin{conjecture}\label{Conjecture Nonchalant Patterns}
	For every avoidable pattern $P$ and any integer $n\geqslant \mu(P)$, the sequence of $P$-nonchalant words over $\mathbb{A}_n$ is infinite and converges to a unique infinite word $\mathcal{N}^{(P)}_n$.
\end{conjecture}
A first attempt in studying extremal $P$-free words was made by Ter-Saakov and Zhang in \cite{Ter-SaakovZhang}, though they focused on a special family of \emph{unavoidable} patterns of the form $P=XY_1XY_2X\cdots XY_tX$, determining the exact number of extremal $P$-free words over any finite alphabet.

Actually, one may go into broader generality and consider similar problems for any \emph{monotone property} of words (sets of words closed under any alphabet permutation and taking factors). One natural example from outside the pattern avoidance setting, are words avoiding \emph{abelian squares} (words of the form $XY$, where $Y$ is any permutation of $X$). It is known that there exist infinitely many abelian square-free words over a $4$-letter alphabet, as conjectured by Erd\H{o}s \cite{Erdos} and proved by Ker\"{a}nen \cite{Keranen}. Ter-Saakov and Zhang found in \cite{Ter-SaakovZhang} the shortest extremal abelian square-free word over four letters: $$\mathtt{123421324321},$$ and conjectured that there are infinitely many of them. On the other hand, it is not hard to check that the number of nonchalant abelian square-free words over any finite alphabet is finite.

\appendix

\section{10 000 iterations of the nonchalant procedure}\label{apA}

We consider a 3-letter alphabet. Let $p$ be the number of positions that the nonchalant algorithm moved back in the $i$-th iteration. Tables \ref{TabNon1} and \ref{TabNon2} contain information about the~first 10~000 iterations of the nonchalant algorithm with initial word $\mathtt{1}$. For example, in the seventh iteration of the~algorithm, the 8-letter long word was obtained by inserting a single letter in the penultimate position of the previous, 7-letter long, word. The~three iterations with the~biggest number of positions moved back were bolded (in fact, these are the only iterations in which the algorithm moved back more than 9 positions (cf. Table \ref{nai}).

Table \ref{T2} contains an example of more detailed common results for various initial words. Let us focus on the distances between consecutive iterations, in which the nonchalant algorithm moved back by exactly four positions (the number 4 was chosen arbitrarily). We present the~number of occurrences of such distances with respect to various initial words (for each initial word we, again, analyse the first 10 000 iterations of nonchalant algorithm). For example, for initial word $\mathtt{1}$, the first occurrence of considered iteration takes place after 207 steps (Table \ref{TabNon1}). Such number of steps between two consecutive iterations does not happen anymore, so for the initial word $\mathtt{1}$ we have a number 1 in the column started by 207. As we can see, the most common distances among considered iterations are about 210-211 steps.

\begin{table}\centering
\begin{tabular}{|r|c||r|c||r|c||r|c||r|c||r|c||r|c||r|c|}\hline
$i$&$p$&$i$&$p$&$i$&$p$&$i$&$p$&$i$&$p$&$i$&$p$&$i$&$p$&$i$&$p$\\\hline
7&1&640&1&1307&2&1965&1&2592&4&3253&1&3861&1&{\bf4436}&{\bf20}\\
25&2&648&2&1338&1&1986&1&2625&1&3256&2&3883&2&4453&1\\
32&1&676&1&1349&7&1994&2&2657&1&3279&4&3890&2&4485&1\\
64&1&698&2&1382&1&2025&1&2662&1&3312&1&3921&1&4490&1\\
69&1&705&2&1387&1&2036&4&2665&2&3344&1&3953&1&4493&2\\
72&2&764&1&1390&2&2040&7&2696&1&3349&1&3958&1&4524&1\\
103&1&769&1&1421&1&2067&2&2728&1&3352&2&3961&2&4556&1\\
135&1&772&2&1453&1&2074&2&2730&2&3383&1&3992&1&4558&2\\
140&1&803&1&1458&1&2105&1&2732&2&3415&1&4013&1&4560&2\\
{\bf143}&{\bf15}&835&1&1461&2&2137&1&2761&2&3417&9&4021&2&4589&2\\
144&2&840&1&1484&4&2142&1&2792&1&3438&2&4052&1&4620&1\\
175&1&843&2&1517&1&2145&2&2824&1&3445&2&4063&4&4652&1\\
207&1&902&1&1549&1&2176&1&2829&1&3476&1&4093&2&4657&1\\
212&1&907&1&1554&1&2197&1&2832&2&3508&1&4100&2&4660&2\\
215&2&910&2&1557&2&2205&2&2863&1&3513&1&4131&1&4691&1\\
246&1&931&2&1588&1&2236&1&2884&1&3516&2&4163&1&4712&1\\
270&4&959&2&1620&1&2247&4&2892&2&3547&1&4168&1&4720&2\\
300&2&966&2&1622&2&2277&2&2923&1&3568&1&4171&2&4751&1\\
307&2&997&1&1624&2&2284&2&2934&4&3576&2&4202&1&4762&4\\
338&1&1029&1&1653&2&2315&1&2964&2&3607&1&4223&1&4792&2\\
370&1&1034&1&1684&1&2347&1&2971&2&3618&4&4231&2&4799&2\\
375&1&1037&2&1716&1&2352&1&3002&1&3648&2&4262&1&4830&1\\
378&2&1068&1&1721&1&2355&2&3034&1&3655&2&4273&4&4862&1\\
409&1&1089&1&1724&2&2386&1&3039&1&3686&1&4274&1&4867&1\\
430&1&1097&2&1755&1&2407&1&3042&2&3718&1&4275&1&4870&2\\
438&2&1128&1&1776&1&2415&2&3073&1&3723&1&4278&2&4901&1\\
469&1&1139&4&1784&2&2446&1&3094&1&3726&2&4296&2&4922&1\\
480&4&1169&2&1815&1&2457&7&3102&2&3757&1&4303&2&4930&2\\
510&2&1176&2&1826&4&2490&1&3133&1&3778&1&4342&1&4961&1\\
517&2&1207&1&1856&2&2495&1&3144&7&3786&2&4347&1&4972&4\\
548&1&1239&1&1863&2&2498&2&3177&1&3817&1&4350&2&4976&7\\\cline{15-16}
580&1&1244&1&1894&1&2529&1&3182&1&3828&4&4381&1\\
585&1&1247&2&1926&1&2561&1&3185&2&3829&1&4413&1\\
588&2&1278&1&1931&1&2566&1&3216&1&3830&1&4418&1\\
619&1&1299&1&1934&2&2569&2&3248&1&3833&2&4421&2\\\cline{1-14}
\end{tabular}
\vskip0.5cm
\caption{Non-zero positions moved back by nonchalant algorithm in iterations 1-5000.}\label{TabNon1}

\end{table}

\begin{table}\centering
\begin{tabular}{|r|c||r|c||r|c||r|c||r|c||r|c||r|c||r|c|}\hline
$i$&$p$&$i$&$p$&$i$&$p$&$i$&$p$&$i$&$p$&$i$&$p$&$i$&$p$&$i$&$p$\\\hline
5003&2&5632&1&6249&2&6852&2&7483&2&8144&1&8835&1&9485&2\\
5010&2&5664&1&6280&1&6883&1&7514&1&8147&2&8843&2&9516&1\\
5041&1&5666&2&6291&4&6904&1&7535&1&8178&1&8874&1&9527&4\\
5073&1&5668&2&6321&2&6912&2&7543&2&8205&2&8885&4&9557&2\\
5078&1&5697&2&6328&2&6943&1&7574&1&8238&1&8915&2&9564&2\\
5081&2&5728&1&6359&1&6954&4&7585&7&8261&3&8922&2&9595&1\\
5112&1&5760&1&6391&1&6984&2&7618&1&8267&2&8953&1&9627&1\\
5133&1&5765&1&6396&1&6991&2&7623&1&8298&1&8985&1&9632&1\\
5141&2&5768&2&6399&2&7022&1&7626&2&8330&1&8990&1&9635&2\\
5172&1&5799&1&6430&1&7054&1&7657&1&8335&1&8993&2&9666&1\\
5183&4&5820&1&6451&1&7059&1&7689&4&8338&2&9024&1&9687&1\\
5213&2&5828&2&6459&2&7062&2&7716&2&8369&1&9045&1&9695&2\\
5220&2&5859&1&6490&1&7093&1&7723&2&8390&1&9053&2&9726&1\\
5251&1&5870&4&6501&7&7114&1&7754&1&8398&2&9084&1&9737&4\\
5283&1&5900&2&6534&1&7122&2&7786&1&8429&1&9123&2&9738&1\\
5288&1&5907&2&6539&1&7153&1&7791&1&8440&4&9130&2&9739&1\\
5291&2&5938&1&6542&2&7164&4&7794&2&8470&2&9161&1&9742&2\\
5322&1&5970&1&6573&1&7168&7&7825&1&8477&2&9193&1&9770&1\\
5343&1&5975&1&6605&1&7195&2&7846&1&8508&1&9198&1&9792&2\\
5351&2&5978&2&6610&1&7202&2&7854&2&8540&1&9201&2&9799&2\\
5382&1&6009&1&6613&2&7233&1&7885&1&8545&1&9232&1&9830&1\\
5393&7&6030&1&{\bf6628}&{\bf20}&7265&1&7896&4&8548&2&9253&1&9862&1\\
5426&1&6038&2&6645&1&7270&1&7926&2&8579&1&9261&2&9867&1\\
5431&1&6069&1&6677&1&7273&2&7933&2&8600&1&9292&1&9870&2\\
5434&2&6080&4&6682&1&7304&1&7964&1&8608&2&9294&1&9901&1\\
5465&1&6084&7&6685&2&7325&1&7996&1&8639&1&9328&4&9922&1\\
5497&1&6111&2&6716&1&7333&2&8001&1&8650&4&9347&2&9930&2\\
5502&1&6118&2&6748&1&7364&1&8004&2&8705&2&9354&2&9961&1\\
5505&2&6149&1&6750&2&7375&4&8035&1&8712&2&9385&1&9972&4\\\cline{15-16}
5528&4&6181&1&6752&2&7405&2&8056&1&8743&1&9417&1\\
5561&1&6186&1&6781&2&7412&2&8064&2&8775&1&9422&1\\
5593&1&6189&2&6812&1&7443&1&8095&1&8780&1&9425&2\\
5598&1&6220&1&6844&1&7475&1&8106&7&8783&2&9456&1\\
5601&2&6241&1&6849&1&7480&1&8139&1&8814&1&9477&1\\\cline{1-14}
\end{tabular}
\vskip0.5cm
\caption{Non-zero positions moved back by nonchalant algorithm in iterations 5001-10000.}\label{TabNon2}
\end{table}

\begin{table}\centering
	\begin{tabular}{|c||cccccccccccccc|}\hline
$N_1$&199&207&208&210&211&233&235&270&314&339&342&345&443&460\\ \hline
	    $\mathtt{1}$&1&1&0&9&4&0&3&1&1&1&3&4&1&0\\
	    $\mathtt{2}$&2&1&0&8&5&0&3&0&1&0&4&4&1&1\\
	    $\mathtt{3}$&1&1&0&8&6&1&1&0&1&1&5&6&0&1\\
	    $\mathtt{13}$&1&1&1&8&6&0&1&0&1&1&5&6&0&1\\
	    $\mathtt{23}$&2&1&0&9&5&0&3&0&1&1&3&4&1&1\\
	    $\mathtt{32}$&2&1&0&8&5&0&3&0&1&0&4&4&1&1\\
	    $\mathtt{3213}$&1&1&1&8&6&0&1&0&1&1&5&6&0&1\\
	    $\mathtt{2313}$&2&0&0&8&5&0&2&0&0&1&5&7&1&1\\
	    $\mathtt{32132}$&1&1&0&9&6&0&2&0&1&0&4&1&0&0\\
	    $\mathtt{2313213}$&1&1&1&8&6&0&1&0&1&1&5&6&0&1\\
	    $\mathtt{231323}$&1&2&0&8&6&0&1&0&1&1&5&6&0&1\\
	    $\mathtt{32132313}$&2&0&0&9&5&0&3&0&0&1&5&7&1&1\\
	    $\mathtt{321323132}$&2&0&0&9&5&0&3&0&0&1&5&7&1&1\\
	    $\mathtt{32132313213}$&1&1&1&8&6&0&1&0&1&1&5&6&0&1\\
	    \hline
	\end{tabular}

\end{table}

	\begin{table}\centering
	\begin{tabular}{|c||cccccc|}\hline
$N_1$&489&544&659&663&688&806\\ \hline
	    $\mathtt{1}$&1&1&1&1&0&0\\
	    $\mathtt{2}$&0&1&0&1&1&0\\
	    $\mathtt{3}$&1&0&0&1&0&0\\
	    $\mathtt{13}$&1&0&0&1&0&0\\
	    $\mathtt{23}$&1&1&0&1&0&0\\
	    $\mathtt{32}$&0&1&0&1&1&0\\
	    $\mathtt{3213}$&1&0&0&1&0&0\\
	    $\mathtt{2313}$&0&0&0&0&0&0\\
	    $\mathtt{32132}$&0&1&0&1&0&1\\
	    $\mathtt{2313213}$&1&0&0&1&0&0\\
	    $\mathtt{231323}$&1&0&0&1&0&0\\
	    $\mathtt{32132313}$&0&0&0&0&0&0\\
	    $\mathtt{321323132}$&0&0&0&0&0&0\\
	    $\mathtt{32132313213}$&1&0&0&1&0&0\\
	    \hline
	    \end{tabular}
    \vskip0.5cm
	   	\caption{Numbers of occurences of distances (first row) between two consecutive iterations in which the nonchalant algorithm moved back exactly four positions, for various initial words (for 10 000 iterations of the algoritm).}\label{T2}
	    \end{table}

\end{document}